\documentclass[a4paper,10pt]{amsart}

\usepackage[english]{babel}
\usepackage[utf8]{inputenx}
\usepackage[T1]{fontenc} 
\usepackage{graphicx}         
\usepackage{hyperref}
\usepackage{tikz-cd}
\usepackage{lmodern}
\usepackage{verbatim}
\usepackage{longtable}
\usepackage{color}
\usepackage{float}
\usepackage[table]{xcolor}

\usepackage{multirow}%
\usepackage{amsmath,amssymb,amsfonts}%
\usepackage{amsthm}%
\usepackage{mathrsfs}%
\usepackage[title]{appendix}%
\usepackage{textcomp}%
\usepackage{manyfoot}%
\usepackage{booktabs}%
\usepackage{algorithm}%
\usepackage{algorithmicx}%
\usepackage{algpseudocode}%
\usepackage{listings}%

\numberwithin{equation}{section}

\title{Class number of real quadratic fields of explicit discriminant}

\author{Riccardo Bernardini}

\address{Dipartimento di Matematica Guido Castelnuovo, Sapienza Università di Roma, Piazzale Aldo Moro 5, 00185 Rome, Italy}
\email{r.bernardini@uniroma1.it}
\date{\today} 



\begin{document}

\newtheorem{theorem}{Theorem}[section]
\newtheorem{definition}[theorem]{Definition}
\newtheorem{lemma}[theorem]{Lemma}
\newtheorem{corollary}[theorem]{Corollary}
\newtheorem{proposition}[theorem]{Proposition}
\newtheorem{example}[theorem]{Example}
\newtheorem{remark}[theorem]{Remark}
\theoremstyle{definition}

\begin{abstract}In this paper, we study the class numbers of a family of real quadratic fields $\mathbb{Q}(\sqrt{D})$ whose discriminants possess a specific exponential structure. More precisely, the discriminant $D$ is given by $(b(4bs+1)^k+s)^2 + (4bs+1)^k$, where $b$, $s$ and $k$ are positive integers, and $D$ is assumed to be square-free. This non-Richaud-Degert family is characterized by a predictable continued fraction expansion for $\sqrt{D}$, which allows for an explicit expression of its fundamental unit.
When the parameters $b$ and $s$ are fixed, the corresponding regulators grow asymptotically like $(\log D)^2$. By analyzing the associated partial twisted zeta functions at $z=0$ and applying algebraic methods inspired by the strategies of Biró, we derive explicit criteria under which the class number $h(D)$ is strictly greater than $1$. Our work contributes to the long-standing mathematical tradition that aims to solve the "class number one problem" for explicit families of real quadratic fields.
\end{abstract}

\maketitle

\section{Introduction}
The class number of a number field measures the failure of its ring of integers to be a principal ideal domain. Despite long investigations, it is still the center of many unanswered questions; for instance, Gauss' famous conjecture asserts the existence of infinitely many real quadratic fields with class number one. 
Gauss' conjecture is notably difficult to attack: as shown by the class number formula, the point is to find large fundamental units in relation to the discriminant. Indeed, if $\mathbb{Q}(\sqrt{D})$ has discriminant $D>0$ and class number $h(D)$, the class number formula and an ineffective theorem by Siegel \cite{Siegel} provide the following inequality
\begin{equation}\label{diseq an 1}
h(D)\gg_\epsilon \frac{D^{\frac{1}{2}-\epsilon}}{\log{\xi_D}},
\end{equation}
where $\xi_D$ is the fundamental unit and $\gg_\epsilon$ means that the left-hand side is greater than the right-hand side multiplied by a positive constant depending on $\epsilon$. 
From \eqref{diseq an 1} we see that, if $\xi_D$ is small with respect to $D$, the class number will eventually be greater than any constant. This happens for all successions of explicit discriminants $D>0$ considered in the literature, and it is thought to be the most likely case. Thus, a natural question arises: given a family of positive fundamental discriminants, at what point does it begin to produce class numbers greater than one? 
This is usually referred to as the \textit{class number one problem} for the given family. 

In the last few decades, a remarkable number of mathematicians have analyzed quadratic discriminants of \textit{Richaud--Degert type} (RD), i.e. of the form $D=n^2+r$ where $D$ is square-free, $r|4n$ and $-n<r\leq n$, reformulating the class number one problem in terms of prime producing polynomials (see Yokoi \cite{Yokoi} or Byeon and Kim \cite{Byeon}). 
However, except for a few cases like \cite[Corollary 3.1]{Byeon} and \cite[Corollary 3.2]{Byeon}, this approach does not allow to know exactly when the class number is one. 

Later, Biró \cite{Biro1,Biro3}, building on a work by Beck \cite{Beck}, used zeta functions related to non-quadratic Dirichlet characters to determine the complete list of RD fundamental discriminants of the form $D=4n^2+1$ and $D=n^2+4$ (known respectively as \textit{Chowla's family} and \textit{Yokoi's family}) with $h(D)=1$. 
His approach was primarily based on \textit{ad hoc} calculations, but it can be adapted to other RD fields.
Indeed, Byeon, Kim and Lee \cite{Byeon2} used his method to investigate the family $D=n^2-4$ and their achievement, together with \cite{Biro1,Biro3} and \cite[Corollary 3.2]{Byeon}, completely settles the class number one problem for RD discriminants with $|r|=1,4$.

This strategy can also be applied to non-RD discriminants. In \cite[Theorem 1.1]{Biro2}, for instance, Biró considered quadratic fields given by the square-free values of
\begin{align}\label{D Biro}
D_{Biro}(b,s,k)=(b(bs+1)^k+s)^2 + 4(bs+1)^k, 
\end{align}
with $b\geq0$, $s>0$, $k\geq 2$. 
If $(1,\omega_D)$ is the standard integral basis of the ring of integers of $\mathbb{Q}(\sqrt{D})$, then $\omega_{D_{Biro}}$ has a known continued fraction, which allows for direct computations, and a large fundamental unit of size $\log{\xi_D}\asymp(\log{D})^2$, as shown in \cite[Satz 1]{HK}. Under the hypothesis $41\cdot61\cdot175\cdot1861|b$, Biró could show that if $h(D)=1$ then $b=0$ and \eqref{D Biro} becomes  $n^2+4$, which had already been studied.

Other non-RD families, with a known continued fraction and explicit fundamental unit, have been found by Mc Laughlin and Zimmer \cite{McZimmer}. Their work provides numerous examples of real quadratic fields, with $\log{\xi_D}\asymp(\log{D})^2$, that can be concretely handled to attack the class number one problem.

In this paper, we aim to extend Biró’s method to a new family of discriminants that are not of RD type, namely those described in \cite[Theorem 6]{McZimmer}: 
\begin{align}\label{D McZim}
D(b,s,k)=(b(4bs+1)^k+s)^2 + (4bs+1)^k
\end{align}
with $b,s,k>0$. We obtain the following theorem.

\begin{theorem}\label{thm 1}
Let $b,s,k\geq 1$ and let $D=D(b,s,k)$ (as in equation \eqref{D McZim}) be square-free. Then, if at least one of the conditions 
\begin{itemize}
    \item[1)] $(b,s)>2$;
    \item[2)] $4bs+1$ is not a prime;
    \item[3)] $D\equiv2\mod{4}$;
    \item[4)] $D\equiv 1\mod{4}$ and $41\cdot61\cdot181\cdot1861|b$
\end{itemize}    
hold, the class number is greater than $1$.
\end{theorem}

The proof will be divided into two parts: in the first one we will prove the theorem if one among $1)$, $2)$ and $3)$ holds, in the second one we will only assume condition $4)$.

The numbers in $4)$ come from the application of the main idea of Biró. Indeed, we will consider pairs $(q,r)$, where $q\neq r$ are odd primes that divide $b$. Through a construction involving a Dirichlet character modulo $q$, we will show that, if $D\equiv d_0\mod{q}$, then it lies in a class $r_0$ modulo $r$ depending only on $d_0$. Therefore, if $D$ is not a square modulo $q$, one can hope that $D$ is a square modulo $r$. If again $D$ is not a quadratic residue modulo $r$, we repeat the argument with a couple $(r,p)$, where $p|b$ and $p\neq r$.
From the combination of various pairs $(q,r)$ with $q,r\in\{41,61,181,1861\}$, we will deduce that, if $h(D)=1$, then $D$ is a square modulo at least one among $41$, $61$, $181$ and $1861$, see Proposition \ref{teorema num class}. 
Anyway, the choice of $(q,r)$ and of the characters is not unique: Theorem \ref{thm 1} is valid also for the sets of integers $\{41,43,181,211\}$ and $\{41,181,211,281\}$.

The strategy adopted has proven to be successful in settling the class number one problem for various RD families, while for \eqref{D Biro} and \eqref{D McZim} it can provide only a partial solution. Thus, our result pursues the long-term task of solving the problem for all the available families in the literature, especially for those like \eqref{D McZim}, i.e. not of RD type and of exponential form. See, for instance, \cite{HK} or \cite{McZimmer}.

There are also other interesting quadratic fields: for example, those given by $D=n^2\pm 4p$ or $D=(p^kq+p+1)^2-4p$ for $p<q$ primes, see Yamamoto \cite{Yamamoto}. Unfortunately, here $\omega_D$ does not have a predictable continued fraction, so it is hard to make explicit manipulations as with \eqref{D McZim}.

\begin{remark}\label{remark n^2+1}
In \eqref{D McZim} we can allow also $k,s=0$. Indeed, if $bsk=0$ then the discriminant $D(b,s,k)$ in \eqref{D McZim} becomes of the form $n^2+1$ a family for which the problem is closed. 
\end{remark}

\section{First part of the proof}\label{First part}

We will prove Theorem \ref{thm 1} assuming the conditions $1)$, $2)$ or $3)$.
We define $\tau:=(4bs+1)$ and take $D$ as in \eqref{D McZim}; then, by \cite[Theorem 6]{McZimmer}, we get 

\begin{equation}\label{expansion 1}
\sqrt{D}=[b\tau^k+s,\overline{2b,2b\tau^{k-1},2b\tau,2b\tau^{k-2}, \ldots, 2b\tau^{k-1},2b,2b\tau^k+2s}].
\end{equation} 

Since we want to deal with both cases $D\equiv 1\mod{4}$ and $D\not\equiv 1\mod{4}$, it is useful to derive the continued fraction of $\frac{1+\sqrt{D}}{2}$ from that of $\sqrt{D}$. We will use the next elementary lemma, which we state without a proof.

\begin{lemma}\label{lemma divisione}
Let $x$ be a real number with continued fraction expansion
$$x=[b_0,\overline{b_1,b_2,\ldots,b_l}].$$
Assume that $b_i\geq3$ for $i\geq 1$ and $b_i$ is even for $i\geq0$. Then

\begin{equation*}
\frac{1+x}{2}=\left[\frac{b_0}{2},\overline{1,1,\frac{b_1-2}{2},1,1,\frac{b_2-2}{2},1,1,\ldots,\frac{b_l-2}{2}}\right].
\end{equation*}
\end{lemma}

In particular, if $D\equiv 1\mod{4}$ with $b\geq2$ and $k\geq1$, then
$$
\frac{1+\sqrt{D}}{2}=[b_0,\overline{b_1,b_2,\ldots,b_{3(2k+1)}}]
$$
with

\begin{align}\label{esp lemma divisione}
\begin{split}
&b_{3(2i-1)}=b\tau^{i-1}-1,\quad b_{3(2i)}=b\tau^{k-i}-1 \qquad (1\leq i\leq k)\\
&b_0=\frac{b\tau^k+s}{2},\quad b_{3(2k+1)}=2b_0-1,\quad \mbox{and}\quad b_i=1\  \mbox{otherwise.}
\end{split}
\end{align}

We first prove the theorem assuming $(b,s)>2$ and let $q$ be a prime dividing $(b,s)$.  

\begin{proof}[Proof of Theorem \ref{thm 1} assuming 1)]
Since $D\equiv 1\mod{q}$, if $q$ is odd, then $D$ is a quadratic residue modulo $q$ and thus the ideal $(q)$ factors into the product of two distinct ideals in $\mathcal{O}_K$. If $q=2$ and $4|(b,s)$, then $D\equiv1\mod{8}$ and $(2)$ splits. Then there exists, in any case, a prime $q|(b,s)$ that splits. If we add the hypothesis $h(D)=1$, then there exists $\omega\in \mathcal{O}_K$ such that $|\omega\bar\omega|=q$. 
Now we want to use the properties of the continued fractions of $\sqrt{D}$ or $\frac{1+\sqrt{D}}{2}$ to get a contradiction; therefore, we distinguish the two cases reasoning as in \cite[§3]{Biro2}. 

For both, we are led to two possible inequalities:
\begin{equation}\label{eq 1 def prop 1}
b\tau^k+s<\sqrt{D}\leq 2q[b\tau^{k-1}+1].
\end{equation}
and
\begin{equation}\label{eq 2 def prop 1}
(q-1)\frac{b_l}{q} <2,
\end{equation}
where $b_l$ is the last term of the period of the continued fraction we are dealing with.
By \eqref{eq 1 def prop 1} and $q|(b,s)$, it follows that
$$
2q[b\tau^{k-1}+1]\leq 2bs\tau^{k-1}+2bs\leq 4bs\tau^{k-1}<\tau^k
$$ 
and thus $b\tau^k+s<\tau^k$, which is absurd. 
Inequality \eqref{eq 2 def prop 1} also gives a contradiction because $q\geq2$ and $k\geq1$.
\end{proof}

Now let $\tau=4bs+1$ not be a prime. Since we need to find the continued fraction of $\frac{1+\sqrt{D}}{2}$ also when $b=1$, we start with a preparatory lemma, of which we omit the proof as it follows the same logic as Lemma \ref{lemma divisione}. 

\begin{lemma}\label{lemma divisione 2}

Being $D$ as in \eqref{D McZim}, if $b=1$, $k\geq2$ and $D\equiv 1\mod{4}$ then
$$
\frac{1+\sqrt{D}}{2}=[b_0,\overline{b_1,b_2,\ldots,b_{6k-1}}]
$$
with
\begin{align*}
&b_{6(i-1)+7}=\tau^{i}-1,\quad b_{6(i-1)+4}=\tau^{k-i}-1 \qquad (1\leq i\leq k-1)\\
&b_0=\frac{\tau^k+s}{2},\quad b_{6k-1}=2b_0-1,\quad b_2,b_{6k-3}=2,\quad \mbox{and}\quad b_i=1\  \mbox{otherwise.}
\end{align*}
If instead $k=1$ then
$$
\frac{1+\sqrt{D}}{2}=\left[\frac{5s+1}{2},\overline{1,2,2,1,5s}\right].
$$
\end{lemma}

\begin{proof}[Proof of Theorem \ref{thm 1} assuming 2)]
Let us suppose $h(D)=1$ and let $q$ be the smallest prime dividing $\tau$; then $q$ is odd and $q\leq \tau^{1/2}$. 
Obviously, $D$ is a square modulo $q$ and therefore, reasoning again as in \cite[§3]{Biro2}, we get for both $D\equiv1\mod{4}$ and $D\not\equiv1\mod{4}$ the inequalities \eqref{eq 1 def prop 1} and
\begin{equation}\label{eq 2 def prop 2}
b\tau^k+s<5.
\end{equation}
Just for $b=k=1$ and $D\equiv 1\mod{4}$, inequality \eqref{eq 1 def prop 1} is replaced by $\tau+s<8q$. Clearly, \eqref{eq 2 def prop 2} is always absurd, while $\tau+s<8q$ is absurd if $\tau\geq 64$ since $q\leq\tau^{1/2}$. Thus, let us divide the remaining part of the proof into cases.

\begin{itemize} 
   \item[1)] 
Assume $b,k=1$. If $\tau\geq 64$ holds, we immediately get a contradiction for every class of $D$ modulo $4$, so if $h(D)=1$ then
$$
\tau \in [9,21,25,33,45,49,57]
$$
or equivalently
$$
s\in [2,5,6,8,11,12,14]
$$
and we check by PARI/GP that all class numbers are greater than one.

   \item[2)] Assume not both $b$ and $k$ are $1$ so that we can focus just on \eqref{eq 1 def prop 1}. If we define $a:=\log_9{6}$, we can show that
\begin{equation}\label{eq 1 def prop 1 bis}
\tau^a\leq \frac{b\tau^k+s}{b\tau^{k-1}+1}
\end{equation}
for all $k$; indeed this is equivalent to $1\leq b\tau^{k-1}(\tau^{1-a}-1)+s\tau^{-a}$
and 
$$
b\tau^{k-1}(\tau^{1-a}-1)+s\tau^{-a}\geq 2(\tau^{1-a}-1)\geq 2(9^{1-a}-1)=1.
$$
Since $\tau\geq9$, on the other hand, we have that $2q\leq 2\tau^{1/2}\leq \tau^a$.
Therefore, by \eqref{eq 1 def prop 1} and \eqref{eq 1 def prop 1 bis}, one has $\tau^a<2q\leq \tau^a$.
\end{itemize}
The proof is complete.
\end{proof}

We end this section by proving that $h(D)>1$ if $D\equiv0\mod{4}$. In order to do this, we introduce the notion of the type of a quadratic irrational.

\begin{definition}
    Let $\xi$ be a quadratic irrational. The unique triple $(a,b,c)\in \mathbb{Z}^3$ such that $(a,b,c)=1$, $a\neq0$ and
    $$
    \xi=\frac{b+\sqrt{b^2-4ac}}{2a}
    $$
    is called the type of $\xi$. Here, the square root is normalized so that it is in $\mathbb{R}_{>0}$ or has a positive imaginary part.
    Clearly, $a\xi^2-b\xi+c=0$.
\end{definition}

The idea is to apply a result of Louboutin \cite[Theorem 3]{Louboutin}, thus we will compute the first entry of the types of the partial quotients of the continued fraction of $\omega_{4D}=\sqrt{D}$. The original formulation of Louboutin does not involve the types, but the connection is clear in a paper by Kawamoto, Kishi and Tomita \cite[§ 3.2]{KaKiTo}, where they also generalized \cite[Theorem 3]{Louboutin}. For convenience, we restate the result here.

    \begin{theorem}[Louboutin]\label{thm Louboutin}
        Let $K$ be the real quadratic field of discriminant $d$ and with Kronecker character $\chi_d$. Let $\omega_d$ have continued fraction
        $$
        \omega_d=[b_0,\overline{b_1,\ldots,b_l}],
        $$
        where $l$ is the period, let
        $$
        \xi_i=[\overline{b_i,\ldots,b_{i+l-1}}]
        $$
        be the $i$-th partial quotient and indicate with $a_i$ the first entry of the type of $\xi_i$.
        Then $K$ has class number one if and only if for any of the primes $p\leq\frac{\sqrt{d}}{2}$ satisfying $\chi_d(p)\neq-1$, there is an $1\leq i\leq \lfloor l/2 \rfloor$ such that $p=a_i$.
    \end{theorem}
    For the computation of the $a_i$, we rely on \cite[Theorem 2.3.5.1.a]{HKbook} and will determine them by induction.
    
    \begin{proposition}\label{prop 1}
        Let $D\not\equiv1\mod{4}$ be as in \eqref{D McZim}, so that $\omega_{4D}=\sqrt{D}$, and square-free. For $1\leq i \leq k$, one has that
        $$
        a_{2i-1}=\tau^{k+1-i}\quad \mbox{and}\quad a_{2i}=\tau^{i}.
        $$
        In particular, the prime $2\neq a_j$ for any $1\leq j\leq 2k$.
    \end{proposition}
    Now we can conclude the first part of the demonstration of Theorem \ref{thm 1}.
    \begin{proof}[Proof of Theorem \ref{thm 1} assuming 3)]
        The thesis follows from Proposition \ref{prop 1}, Theorem \ref{thm Louboutin} and the trivial fact that $\chi_{4D}(2)=0$.
    \end{proof}

    \begin{remark}
        In the case $D\equiv 1\mod{4}$, computational evidence suggests that the values of the $a_i$ are not as easily predictable as in Proposition \ref{prop 1}. As a consequence, it seems complicated to make use of Theorem \ref{thm Louboutin}.
    \end{remark}

\section{Second part of the proof} \label{proof thm 1}

Before going into the proof, we establish some notation and briefly sketch the main points of the reasoning.
Throughout this section, let $q$ be an odd prime and let $\chi$ be a Dirichlet character modulo $q$.
We indicate by $K$ the field $\mathbb{Q}(\sqrt{D})$ and by $P(K)$ the set of principal ideals of the ring of integers. Moreover, we indicate by $\mathbb{Q}_\chi$ and $\mathcal{O}_\chi$ respectively the extension of $\mathbb{Q}$ given by the values of $\chi$ and its ring of integers. We introduce the following algebraic integer
$$
m_\chi:=\sum_{a=0}^q a\chi(a).
$$
We also assume that $q|b$, and that $D$ is as in \eqref{D McZim} and square-free with $b+s$ even. The latter condition is equivalent to $D\equiv 1\mod{4}$, so that $D$ is the fundamental discriminant. 
We recall that $\tau=(4bs+1)$ and that
$$
\omega_{D}=[b_0,\overline{b_1,b_2,\ldots,b_{6k+3}}],
$$
where the $b_i$ are given by equations \eqref{esp lemma divisione}.
Inspired by the strategy of Biró, we want to apply \cite[Theorem 1]{BiroGranville} in order to derive a formula for the special value at $z=0$ of
$$
\zeta_{P(K)}(z):=\sum_{I\in P(K)}\frac{\chi(N(I))}{N(I)^z},
$$
clearly after having extended $\zeta_{P(K)}$ meromorphically. Indeed, we will prove in Corollary \ref{cor 3.3} that
$$
q^2\zeta_{P(K)}(0,\chi)=B_\chi(s)+R_k,
$$
see Proposition \ref{prop 4.1} for the definitions of $B_\chi(s)$ and $R_k$.
The further hypothesis of $h(D)=1$ gives
$$
 q^2\zeta_{P(K)}(0,\chi)=B_\chi(s)+R_k\in m_\chi\mathcal{O}_\chi.
$$
Let $r$ be a prime different from $q$ such that $r|b$ and $m_\chi\in I$ for a prime ideal $I$ of $\mathcal{O}_\chi$ lying over $r$; thus
\begin{equation}\label{eq mod r}
B_\chi(s)\equiv 0\mod{r}.
\end{equation}

If $s\equiv s_0\mod{q}$, then $B_\chi(s)=B_\chi(s_0)$, and we will see that equation \eqref{eq mod r} forces $s$ to be in a uniquely determined class modulo $r$. Since $D\equiv s^2+1\mod{b}$, if $D$ is not a square modulo $q$, one can check whether $D$ is a square modulo $r$.

At this point, we start with some preparatory computations that we will need to apply \cite[Theorem 1]{BiroGranville}.
We set $\alpha=\omega_D-b_0$, $\frac{p_j}{q_j}=[0,b_1,\ldots,b_j]$, $\alpha_j=p_j-q_j\alpha$ and $\alpha_0=-\alpha$.
By \eqref{esp lemma divisione}, we have the equivalences modulo $q$ 
$$
p_i\equiv\begin{cases}
    0&\quad i\equiv0\mod{6}\\
    1&\quad i\equiv1\mod{6}\\
    1&\quad i\equiv2\mod{6}\\
    0&\quad i\equiv3\mod{6}\\
    1&\quad i\equiv4\mod{6}\\
    1&\quad i\equiv5\mod{6}
\end{cases}\qquad\text{and}\qquad
q_i\equiv\begin{cases}
    1&\quad i\equiv0\mod{6}\\
    1&\quad i\equiv1\mod{6}\\
    2&\quad i\equiv2\mod{6}\\
    -1&\quad i\equiv3\mod{6}\\
    1&\quad i\equiv4\mod{6}\\
    0&\quad i\equiv5\mod{6}
\end{cases},
$$
with the only exceptions of $p_{6k+3}=s$ and $q_{6k+3}=2s-1$.

As in \cite{Biro2}, we introduce the quadratic forms
$$Q_j(x,y)=(\alpha_{j-1}x+\alpha_jy)(\bar\alpha_{j-1}x+\bar\alpha_jy),$$
and if we write
$$Q_j(x,y)=A_jx^2+B_jxy+C_jy^2,$$
we get
\begin{align*}
& A_1=\alpha\bar\alpha,\qquad B_1=2q_1\alpha\bar\alpha-p_1(\alpha+\bar\alpha),\\
&A_j=p_{j-1}^2-p_{j-1}q_{j-1}(\alpha+\bar\alpha)+q_{j-1}^2\alpha\bar\alpha \qquad(j>1),\\
&B_j=2p_{j-1}p_j+2q_{j-1}q_j\alpha\bar\alpha-(p_{j-1}q_j+p_jq_{j-1})(\alpha+\bar\alpha)\qquad (j>1),\\
&C_j=p_j^2-p_jq_j(\alpha+\bar\alpha)+q_j^2\alpha\bar\alpha\qquad (j\geq 1).
\end{align*}
From the congruences $-(\alpha+\bar\alpha)\equiv s-1\mod{q}$ and $ \alpha\bar\alpha\equiv -s/2\mod{q}$,
we derive the next congruences modulo $q$.

\begin{center}
\begin{longtable}{|>{\columncolor{gray!10}}c|c|c|c|c|c|c|}
\caption{}\\
\hline
     \rowcolor{gray!10}   &  $0$  &  $1$  &  $2$  &  $3$  &  $4$  &  $5$  \\
        \hline
    $A_j$ &  $1$  &  $-s/2$  &  $s/2$  &  $-1$  &  $-s/2$  &  $s/2$\\
    \hline
    $B_j$ &  $s-1$  &  $-1$  &  $s-1$  &  $s+1$  &  $1$  &  $s+1$\\
    \hline
    $C_j$ & $-s/2$  &  $s/2$  &  $-1$  &  $-s/2$  &  $s/2$  &  $1$\\
    \hline
\end{longtable}
\end{center}
In the top row we put the congruence class modulo $6$ of the index $j$. The only exceptions are
$$B_{6k+3}=-s+1\quad\mbox{and}\quad C_{6k+3}=\frac{s}{2}.$$

Define
\begin{align*}
    Q_1(U,V)&=-\frac{s}{2}U^2-UV+\frac{s}{2}V^2;\\
    Q_2(U,V)&=\frac{s}{2}U^2+(s-1)UV-V^2;\\
    Q_3(U,V)&=-U^2+(s+1)UV-\frac{s}{2}V^2;\\
    Q_4(U,V)&=-\frac{s}{2}U^2+UV+\frac{s}{2}V^2;\\
    Q_5(U,V)&=\frac{s}{2}U^2+(s+1)UV+V^2;\\
    Q_6(U,V)&=U^2+(s-1)UV-\frac{s}{2}V^2.    
\end{align*}

\begin{proposition}\label{prop 4.1}
If $(q,2D)=1$ and $\chi$ is an odd character of conductor $q$ and order greater than $2$, then $\zeta_{P(K)}$ extends meromorphically to the whole complex plane and
\begin{equation*}
q^2\zeta_{P(K)}(0,\chi)=kA_\chi(s)+B_\chi(s)+R_k,
\end{equation*}
where 
\begin{align*}
A_\chi(s)=&4\sum_{0\leq U,V}^{q-1}UV\sum_{j=2}^4(-1)^j\chi(Q_j(U,V))+2\sum_{0\leq U,V}^{q-1}V^2\chi(Q_3(U,V)),\\
B_\chi(s)=&\sum_{0\leq U,V}^{q-1}UV[4\chi(Q_2(U,V))+2\chi(Q_4(U,V))]\\
+&\sum_{0\leq U,V}^{q-1}V^2\left[\chi(Q_2(U,V))-\chi(Q_4(U,V))+(1-s)\chi(Q_3(U,V))\right]
\end{align*}
and 
$$
R_k=\sum_{0\leq U,V}^{q-1}(V^2-qV)\left[\left(\frac{\tau^k-1}{4s}+b\tau^k\right)\chi(Q_6(U,V))-\frac{\tau^k-1}{4s}\chi(Q_3(U,V))\right].
$$
\end{proposition}

\begin{proof}
The thesis follows from \cite[Theorem 1]{BiroGranville}; indeed, if $A_j,B_j,C_j$ are as above,
\begin{align*}
\frac{q^2\zeta_{P(K)}(0,\chi)}{2}=&k\sum_{0\leq U,V}^{q-1}UV\sum_{j=1}^6(-1)^j\chi(Q_j(U,V))\\
+&\sum_{0\leq U,V}^{q-1}UV[\chi(Q_2(U,V))+\chi(Q_6(U,V))-\chi(Q_1(U,V))]\\
+&q^2 \chi(-d)\left(\frac{d}{q}\right)\tau(\chi)^2\frac{L(2,\bar\chi^2)}{\pi^2}\sum_{j=1}^{6k+3}\frac{(-1)^jb_j\bar\chi(A_j)}{2}.
\end{align*}
The last term can be rewritten through \cite[Proposition 6.1]{BiroGranville} as
$$
\sum_{j=1}^{6k+3}\frac{(-1)^jb_j}{2}\sum_{0\leq U,V}^{q-1}\chi(Q_{\bar{j}}(U,V))\left(V^2-qV\right),
$$
where $\bar j$ is the class of $j$ modulo $6$ except for $Q_{6k+3}(U,V)=-Q_6(U,V)$. After a straightforward computation, we get

\begin{align*}
    &q^2\zeta_{P(K)}(0,\chi)=2k\sum_{0\leq U,V}^{q-1}UV\sum_{j=1}^6(-1)^j\chi(Q_j(U,V))\\
+&k\sum_{0\leq U,V}^{q-1}(V^2-qV)[\chi(Q_2(U,V))+\chi(Q_3(U,V))+\chi(Q_4(U,V))]\\
-&k\sum_{0\leq U,V}^{q-1}(V^2-qV)[\chi(Q_1(U,V))+\chi(Q_5(U,V))+\chi(Q_6(U,V))]\\
+&2\sum_{0\leq U,V}^{q-1}UV[\chi(Q_2(U,V))+\chi(Q_6(U,V))-\chi(Q_1(U,V))]\\
+&\sum_{0\leq U,V}^{q-1}(V^2-qV)\left[\chi(Q_2(U,V))-\chi(Q_1(U,V))+(s-1)\chi(Q_6(U,V))\right]\\
+&R_k
\end{align*}
Now  we work on the term multiplied by $k$, i.e.
\begin{align}
&2\sum_{0\leq U,V}^{q-1}UV\sum_{j=1}^6(-1)^j\chi(Q_j(U,V)) \label{prop 3.1 eq 1}\\
+&\sum_{0\leq U,V}^{q-1}(V^2-qV)[\chi(Q_2(U,V))+\chi(Q_3(U,V))+\chi(Q_4(U,V))]\notag\\
-&\sum_{0\leq U,V}^{q-1}(V^2-qV)[\chi(Q_1(U,V))+\chi(Q_5(U,V))+\chi(Q_6(U,V))]\notag.
\end{align}
Since $\chi$ has order greater than $2$
\begin{equation*}
    \sum_{0\leq U<q}U\chi(U^2)=0,
\end{equation*}
from which one also deduces
    \begin{equation}\label{prop 3.1 eq 2}
        \sum_{0\leq U,V}^{q-1}V\chi(Q_j(U,V))=\sum_{0\leq U,V}^{q-1}U\chi(Q_j(U,V))=0
    \end{equation}
for any $j$.

Moreover, it is easy to show that
\begin{equation}\label{prop 3.1 eq 3}
    \sum_{0\leq U,V}^{q-1}V^2[\chi(Q_4(U,V))-\chi(Q_1(U,V))]=0,    
\end{equation}
 while the changes of coordinates $U\to U+\frac{2}{s}V$ and $U\to U+sV$ in the left-hand sides give respectively
 \begin{align}
    &\sum_{0\leq U,V}^{q-1}V^2\chi(Q_2(U,V))=\sum_{0\leq U,V}^{q-1}V^2\chi(Q_5(U,V))\label{prop 3.1 eq 4}\\
    &\sum_{0\leq U,V}^{q-1}V^2\chi(Q_3(U,V))=-\sum_{0\leq U,V}^{q-1}V^2\chi(Q_6(U,V))\label{prop 3.1 eq 5}.
\end{align}
Notice that we inverted $s$ modulo $q$ because we can assume $(s,q)=1$; otherwise, Theorem \ref{thm 1} is trivially true by condition $1)$ because $q|b$ and it is an odd prime.

By equations \eqref{prop 3.1 eq 2}, \eqref{prop 3.1 eq 3}, \eqref{prop 3.1 eq 4} and \eqref{prop 3.1 eq 5}, we can rewrite \eqref{prop 3.1 eq 1} as
\begin{equation}\label{prop 3.1 eq 6}
2\sum_{0\leq U,V}^{q-1}UV\sum_{j=1}^6(-1)^j\chi(Q_j(U,V))
+2\sum_{0\leq U,V}^{q-1}V^2\chi(Q_3(U,V)).
\end{equation}
Another change of variables proves
\begin{align*}
    \sum_{0\leq U,V}^{q-1}UV\chi(Q_4(U,V))=
       &-\sum_{0\leq U,V}^{q-1}UV\chi(Q_1(U,V))\\
    \sum_{0\leq U,V}^{q-1}UV\chi(Q_2(U,V))=
        &\sum_{0\leq U,V}^{q-1}UV\chi(Q_6(U,V))-q\sum_{0\leq U,V}^{q-1}U\chi(Q_6(U,V))\\
       \overset{\eqref{prop 3.1 eq 2}}{=}&\sum_{0\leq U,V}^{q-1}UV\chi(Q_6(U,V)),
\end{align*}
and, similarly,
\begin{equation*}
    \sum_{0\leq U,V}^{q-1}UV\chi(Q_3(U,V))=\sum_{0\leq U,V}^{q-1}UV\chi(Q_5(U,V)).
\end{equation*}

We can then rearrange equation \eqref{prop 3.1 eq 6} into
\begin{equation*}
    4\sum_{0\leq U,V}^{q-1}UV\sum_{j=2}^4(-1)^j\chi(Q_j(U,V))+2\sum_{0\leq U,V}^{q-1}V^2\chi(Q_3(U,V)).
\end{equation*}
The term $B_\chi(s)$ can be treated analogously.
\end{proof}

\begin{lemma}\label{lemma 3.2}
    Under the hypothesis of the previous proposition, $A_\chi(s)=0$ for any residue class of $s$ modulo $q$.
\end{lemma}

\begin{proof}
    We transform $Q_2(U,V)$ and $Q_4(U,V)$ into $Q_3(U,V)$ using respectively the changes of coordinates
    $$
    \begin{cases}
        U=\overline V\\
        V=\overline{U}-\overline{V}-q\left\lfloor\frac{\overline{U}-\overline{V}}{q}\right\rfloor
    \end{cases}
    \quad\text{and}\quad
    \begin{cases}
        U=\overline{V}-\overline{U}-q\left\lfloor\frac{\overline{V}-\overline{U}}{q}\right\rfloor\\
        V=\overline{U}
    \end{cases},
    $$
    obtaining, after renaming the variables,
    $$
    A_\chi(s)=2\sum_{0\leq U,V}^{q-1}\left[UV-\frac{V^2}{2}-qV\left\lfloor\frac{U-V}{q}\right\rfloor-U^2-qU\left\lfloor\frac{V-U}{q}\right\rfloor\right]\chi(Q_3(U,V)).
    $$
    Now we turn $Q_3(U,V)$ into $-Q_3(U,V)$ through
    $$
    \begin{cases}
        U=\overline{U}-\overline{V}-q\left\lfloor\frac{\overline{U}-\overline{V}}{q}\right\rfloor\\
        V=2\overline{U}-\overline{V}-q\left\lfloor\frac{2\overline{U}-\overline{V}}{q}\right\rfloor
    \end{cases}
    $$
    getting
    \begin{align*}
        &A_\chi(s)=-2\sum_{0\leq U,V}^{q-1}\left[UV-\frac{V^2}{2}-qV\left\lfloor\frac{U-V}{q}\right\rfloor-U^2-qU\left\lfloor\frac{V-U}{q}\right\rfloor\right]\chi(Q_3(U,V))\\
        &-2\sum_{0\leq U,V}^{q-1}\left\lfloor\frac{2U-V}{q}\right\rfloor\left[\frac{q^2}{2}\left\lfloor\frac{2U-V}{q}\right\rfloor-q^2\left\lfloor\frac{U-V}{q}\right\rfloor+q^2\left\lfloor\frac{-U}{q}\right\rfloor\right]\chi(Q_3(U,V)).
    \end{align*}
    Thus, we want to prove that
    \begin{equation}\label{eq brutta}
        \sum_{0\leq U,V}^{q-1}\left\lfloor\frac{2U-V}{q}\right\rfloor\left[\frac{q^2}{2}\left\lfloor\frac{2U-V}{q}\right\rfloor-q^2\left\lfloor\frac{U-V}{q}\right\rfloor+q^2\left\lfloor\frac{-U}{q}\right\rfloor\right]\chi(Q_3(U,V))
    \end{equation}
    is $0$.
    Here, the sum can be restricted to $2U<V$ (which also implies $U<V$) and to $2U\geq V+q$ (which also implies $V<U$).
    Thus, equation \eqref{eq brutta} becomes
    \begin{equation*}
        \sum_{\substack{0\leq U,V\\ 2U<V}}^{q-1}\left[\frac{q^2}{2}-q^2-q^2\left\lfloor\frac{-U}{q}\right\rfloor\right]\chi(Q_3(U,V))
       +\sum_{\substack{0\leq U,V\\ V+q\leq 2U}}^{q-1}\left[-\frac{q^2}{2}\right]\chi(Q_3(U,V)).
    \end{equation*}
    and then
        \begin{equation*}
        -\sum_{V=1}^{q-1}\frac{q^2}{2}\chi\left(-\frac{s}{2}V^2\right)+
        \sum_{\substack{0\leq U,V\\ 0<2U<V}}^{q-1}\frac{q^2}{2}\chi(Q_3(U,V))
       -\sum_{\substack{0\leq U,V\\ V+q\leq 2U}}^{q-1}\frac{q^2}{2}\chi(Q_3(U,V)).
    \end{equation*}
    Since the order of $\chi$ is greater than $2$, the first sum is $0$, so that we just need to prove that
    \begin{equation}\label{eq brutta 1}
        \sum_{\substack{0\leq U,V\\ 0<2U<V}}^{q-1}\chi(Q_3(U,V))
       -\sum_{\substack{0\leq U,V\\ V+q\leq 2U}}^{q-1}\chi(Q_3(U,V))=0.
    \end{equation}
    With the change of coordinates $(U,V)\to (q-U,q-V)$ in the first sum, one gets
    \begin{align*}
    \sum_{\substack{0< U,V\\ V+q<2U}}^{q-1}\chi(Q_3(U,V))=&\sum_{\substack{0\leq U,V\\ V+q<2U}}^{q-1}\chi(Q_3(U,V))-\sum_{\substack{0< U\\ q<2U}}^{q-1}\chi(-U^2)\\
    =&\sum_{\substack{0\leq U,V\\ V+q\leq 2U}}^{q-1}\chi(Q_3(U,V))-\sum_{\substack{0< U\\ q<2U}}^{q-1}\chi(-U^2)-\sum_{\substack{0\leq U,V\\ V+q= 2U}}^{q-1}\chi(U^2);
    \end{align*}
    therefore, equation \eqref{eq brutta 1} is equivalent to
    $$
    \sum_{\substack{0< U\\ q<2U}}^{q-1}\chi(U^2)=\sum_{\substack{0\leq U,V\\ V+q= 2U}}^{q-1}\chi(U^2)
    $$
     and consequently to
    $$
    \sum_{U=\frac{q+1}{2}}^{q-1}\chi(U^2)=\sum_{U=\frac{q+1}{2}}^{q-1}\chi(U^2),
    $$
    which is obviously true.
\end{proof}

\begin{corollary}\label{cor 3.3}
    If $(q,2D)=1$ and $\chi$ is an odd character of conductor $q$ and order greater than $2$, then $\zeta_{P(K)}$ extends meromorphically to the whole complex plane and
\begin{equation*}
q^2\zeta_{P(K)}(0,\chi)=B_\chi(s)+R_k,
\end{equation*}
where $B_\chi(s)$ and $R_k$ are as in Proposition \ref{prop 4.1}.
\end{corollary}

The next lemma is analogous to \cite[Lemma 2.5]{Biro2}.

\begin{lemma}\label{lem 2.4.}
Let $D\equiv1\mod{4}$ as in \eqref{D McZim} and square-free. Let $\chi$ be an odd primitive character of conductor $q>2$ with $(q,2D)=1$. Let us assume that $h(D)=1$.
Then $m_\chi\neq 0$ and 
$$
q\zeta_K(0,\chi)\in m_\chi\mathcal{O}_\chi.
$$
\end{lemma}

We can now prove a lemma that represents a crucial step in our proof.

\begin{lemma}\label{lemma 3.5}
If the hypotheses of Lemma \ref{lem 2.4.} hold with $\chi$ being of order greater than $2$, and if, moreover, there exists a prime ideal $I$ of $\mathcal{O}_\chi$ containing $m_\chi$ and an odd prime $r$ such that $r|b$, then $B_\chi(s)\in I$.
\end{lemma}

\begin{proof}
Lemma \ref{lem 2.4.} implies $q\zeta(0,\chi)\in I$ and consequently $q^2\zeta(0,\chi)\in I$. Since $(r)\subset I$ and we can assume $(s,r)=1$, then $B_\chi(s)\equiv q^2\zeta(0,\chi)\equiv 0\mod{I}$.
\end{proof}

Suppose now that Lemma \ref{lemma 3.5} holds for an odd prime $r\neq q$ and write $s=Pq+s_0$ with $0\leq s_0<q$. Then
\begin{align*}
    0&\equiv B_\chi(s)\equiv B_\chi(s_0)-PqC_\chi(s_0)\\
    \Longrightarrow\ Pq&\equiv \frac{B_\chi(s_0)}{C_\chi(s_0)}\mod{I},
\end{align*}
where 
$$C_\chi(s):=\sum_{0\leq U,V}^{q-1}V^2\chi(Q_3(U,V)).$$
Therefore, if $C_\chi(s_0)\not\equiv 0 \mod{I}$,
$$
s\equiv\frac{B_\chi(s_0)}{C_\chi(s_0)}+s_0\mod{I}.
$$
If $I$ is a prime ideal over $r$ such that $\mathcal{O}_\chi/I=\mathbb{F}_r$, the residue class modulo $q$ of $s$ determines uniquely a residue class modulo $r$.
If the order of $\chi$ is $n>2$, then $\mathbb{Q}_\chi$ is a cyclotomic extension of degree $\varphi(n)$ and, if $r\equiv1\mod{n}$, a classical result shows that
$$
(r)=(r,\zeta-c_1)\ldots(r,\zeta - c_{\varphi(n)}),
$$
where $\zeta$ is a fixed primitive $n$-th root of unity, $\varphi$ is Euler's function and $c_i$ runs among the integers of order $n$ modulo $r$.

For various odd primes $q$, we now construct characters $\chi$ modulo $q$ such that $m_\chi\in I$ for a prime ideal of $\mathcal{O}_\chi$ as above.
The next table, where $\gamma$ is the generator of $\mathbb{F}_q^*$ such that $\chi(\gamma)=\zeta$, summarizes our choices.

\begin{center}
\begin{longtable}{|>{\columncolor{gray!10}}c|c|c|c|c|c|}
\caption{}\\
\hline

      \rowcolor{gray!10}{}{} &  $q$  &  $r$  &  $n$  &  $\gamma$  &  $I$    \\
     
        \hline
        
    $\chi_1$ &  $61$  &  $41$  &  $20$  &  $2$  &  $(41,\zeta-33)$  \\
    
    \hline
    
    $\chi_2$ &  $61$  &  $1861$  &  $60$  &  $2$  &  $(1861,\zeta-1833)$  \\
    
    \hline
    
    $\chi_3$ & $181$  &  $41$  &  $20$  &  $2$  &  $(41,\zeta-20)$  \\
    
    \hline
    
    $\chi_4$ & $181$  &  $61$  &  $60$  &  $2$  &  $(61,\zeta-10)$  \\

    \hline
\end{longtable}
\end{center}

This proves the following proposition.

\begin{proposition}\label{teorema num class}
Let $b,s,k>0$ such that $D$ is square-free and let $h(D)=1$. Suppose, moreover, that $41\cdot61\cdot181\cdot 1861|b$. Then $D$ is a square for at least one of the following moduli: $q=41,61,181,1861$.
\end{proposition}

As anticipated in the introduction, Proposition \ref{teorema num class}, and consequently Theorem \ref{thm 1}, also works with other choices. These are shown in the next two tables: 

\newpage

\begin{center}
\begin{longtable}{|>{\columncolor{gray!10}}c|c|c|c|c|c|}
\caption{}\\
\hline

      \rowcolor{gray!10}{}{} &  $q$  &  $r$  &  $n$  &  $\gamma$  &  $I$    \\
     
        \hline
        
    $\chi_1$ &  $43$  &  $211$  &  $42$  &  $3$  &  $(211,\zeta-26)$  \\
    
    \hline
    
    $\chi_2$ &  $181$  &  $41$  &  $20$  &  $2$  &  $(41,\zeta-20)$  \\
    
    \hline
    
    $\chi_3$ & $211$  &  $41$  &  $10$  &  $2$  &  $(41,\zeta-23)$  \\
    
    \hline
    
    $\chi_4$ & $211$  &  $181$  &  $30$  &  $2$  &  $(181,\zeta-36)$  \\

    \hline
\end{longtable}
\end{center}

or

\begin{center}
\begin{longtable}{|>{\columncolor{gray!10}}c|c|c|c|c|c|}
\caption{}\\
\hline

      \rowcolor{gray!10}{}{} &  $q$  &  $r$  &  $n$  &  $\gamma$  &  $I$    \\
    
    \hline
    
    $\chi_2$ &  $181$  &  $41$  &  $20$  &  $2$  &  $(41,\zeta-20)$  \\
    
    \hline
    
    $\chi_3$ & $211$  &  $41$  &  $10$  &  $2$  &  $(41,\zeta-23)$  \\
    
    \hline
    
    $\chi_4$ & $211$  &  $181$  &  $30$  &  $2$  &  $(181,\zeta-36)$  \\

    \hline
        
    $\chi_1$ &  $211$  &  $281$  &  $14$  &  $2$  &  $(281,\zeta-116)$  \\
    \hline 
\end{longtable}
\end{center}

\begin{remark}
    About the content of Lemma \ref{lemma 3.2}, we could not find a reason \textit{a priori} for $A_\chi(s)=0$. If this fortunate fact had not occurred, the linear dependence in $k$ would have implied congruence conditions on $k$ in the statement of Proposition \ref{teorema num class}.
\end{remark}

We now conclude the demonstration of Theorem \ref{thm 1}.

\begin{proof}[Proof of Theorem \ref{thm 1} assuming 4)]
Take $b\in\mathbb{Z}_{>0}$ and $q\in\{41,61,181,1861\}$ such that $D$ is a square modulo $q$. Then the ideal $(q)$ factors into two prime ideals, not necessarily distinct. That, together with the hypothesis $h(D)=1$, implies the existence of $\omega\in\mathcal{O}_K$ of norm $q$.
Since $D\equiv1\mod{4}$, let $A,B\in\mathbb{Z}$ be such that 
$$
\omega=A-B\frac{1+\sqrt{D}}{2}.
$$
From the equality $|\omega\bar\omega|=q$, one proceeds exactly as in \cite[§3]{Biro2}, obtaining a contradiction.
Therefore, the class number is greater than $1$.
\end{proof}

\subsection*{Acknowledgments}
The author was supported by the research project ``Avvio alla Ricerca 2024'' funded by Sapienza University of Rome, grant D.R. N.1179/2024, project B83C24006550001. The author is also a member of the INdAM group GNSAGA.

Part of this work was carried out during a visit by the author to the Faculty of Mathematics and Physics at Charles University. The author is grateful for the warm hospitality and stimulating research environment.

\bibliographystyle{plain}
\bibliography{sn-bibliography}

\end{document}